\documentclass[11pt]{amsart}
\usepackage[margin=1.2in]{geometry}
\usepackage{amsmath,amssymb}
\usepackage{graphicx}
\usepackage[pdfstartview=FitH]{hyperref}

\newcommand{\N}{{\mathbb{Z}} _{> 0} }
\newcommand{\Zp} {\Z _ {\ge 0} }
\newcommand{\Z}{\mathbb{Z}}
\newcommand{\C}{\mathbb{C}}
\newcommand{\gw}{\mathcal GW_3}
\newcommand{\pp}{\mathcal P}
\newcommand{\h}{\mathbf{h}}
\renewcommand{\c}{\mathbf{c}}
\newcommand{\iu}{\mathrm{i}\mkern1mu}
\newcommand{\normord}[1]{:\mathrel{\mkern2mu #1 \mkern2mu}:}
\def\vpr{v_{p,r}}
\def\<{\langle}
\def\>{\rangle}
\def\l{\lambda}
\renewcommand{\Re}{\operatorname{Re}}
\renewcommand{\Im}{\operatorname{Im}}
\def \ch{\mbox{char}_q}

\newtheorem{theorem}{Theorem}[section]
\newtheorem*{theorem*}{Theorem}
\newtheorem{corollary}[theorem]{Corollary}
\newtheorem{lemma}[theorem]{Lemma}

\newtheorem{remark}[theorem]{Remark}
\newtheorem{example}[theorem]{Example}
\newtheorem{definition}[theorem]{Definition}
\newtheorem{proposition}[theorem]{Proposition}
\def \a{\alpha }
\def \b {\beta}

\def \d {\delta}

\def \l{\lambda }
\def \L{\Lambda }
\def \lm{\overline\l}
\def \m{\mu}
\def \o{\omega}

\newcommand{\bea}{\begin{eqnarray}}
\newcommand{\eea}{\end{eqnarray}}
\newcommand{\bean}{\begin{eqnarray*}}
\newcommand{\eean}{\end{eqnarray*}}

\begin{document}

\title{Galilean $W_3$ algebra}
\author[]{Gordan Radobolja}
\curraddr{Faculty of Science, University of Split, Ru\dj era Bo\v{s}kovi\'{c}a 33,
21 000 Split, Croatia }
\email{gordan@pmfst.hr}
 
\date{\today}
\begin{abstract}
Galilean $W_3$ vertex operator algebra $\gw(c_L,c_M)$ is constructed as a universal enveloping vertex algebra of certain non-linear Lie conformal algebra. It is proved that this algebra is simple by using determinant formula of the vacuum module. Reducibility criterion for Verma modules is given, and the existence of subsingular vectors demonstrated. Free f{}ield realisation of $\gw(c_L,c_M)$ and its highest weight modules is obtained within a rank 4 lattice VOA.
\end{abstract}
\maketitle

\section{Introduction}
Galilean $W$--algebras have been studied extensively by physicists in the past decade (see for example \cite{A}, \cite{AB}, \cite{B}, \cite{BJMN}, \cite{HR}, \cite{RR}, \cite{S} and references therein). Given an inf{}inite-dimensional $W$-algebra with generators $W_1,\ldots,W_k$ of conformal weights $w_1,\ldots,w_k$ the associated Galilean algebra is generated by $W_1',\ldots,W_k',\overline{W}_1,\ldots,\overline{W}_k$ of conformal weights $w_1,\ldots,w_k,w_1,\ldots,w_k$, such that $\<\overline{W}_1,\ldots,\overline{W}_k\>$ is a commutative subalgebra on which all $W_i'$ act. Moreover, the relations between $W_i'$ and $W_j'$ as well as relations between $W_i'$ and $\overline{W}_j$ resemble the original relations between $W_i$ and $W_j$. This new algebra is obtained through a process called \textit{Galilean contraction}. Roughly speaking, one considers a tensor product of two copies of the original algebra (with arbitrary central charges) and takes a \textit{non-relativistic limit} (cf.\ \cite{RR}).

The most basic example is a Galilean conformal algebra (GCA), also known as BMS$_3$-algebra (Bondi-Metzner-Sachs) which comes from contracting the Virasoro algebra. See for example \cite{BH}.
In mathematical literature GCA f{}irst appeared in \cite{DZ} where it was called the $W(2,2)$ (Lie and vertex operator) algebra. Here "(2,2)" denotes conformal weights of two generators. This algebra is constructed by adjoining to the Virasoro algebra its "commutative double", i.e.\ it is a direct sum of (either Lie or vertex operator algebra) $\operatorname{Vir}$ and its adjoint representation. This is analogous to construction of Takif{}f algebras in f{}inite-dimensional case: $\operatorname{Vir}\otimes_\C(\C[x]/(x^2))$ with brackets $[a\otimes x^i,b\otimes x^j]=[a,b]\otimes x^{i+j}$, $a,b\in\operatorname{Vir}$. Free f{}ield realisation of GCA was obtained by means of $\beta\gamma$ system in \cite{BJMN}, albeit only for central charge $c_L=26$. Bosonic free f{}ield realisation for arbitrary non-zero central charge was later obtained in \cite{AR1}-\cite{AR3} and its representation theory has been developed in many papers. We recall the most important results in Subsection \ref{ff1}.

Galilean $W_3$ or BMS$_3$-$W_3$ algebra was originally introduced in \cite{AB}. 
In physics sense, the construction of this algebra follows the same prescription as GCA (cf.\ \cite{RR}) -- contraction of a tensor product of two copies of Zamolodchikov's $W_3$ algebra. Free f{}ield realisation for central charge $c_L=100$ was obtained by double $\beta\gamma$ system in \cite{BJMN}. Mathematically however, things are more complicated. First of all, $W_3$ is not a Lie algebra (quadratic terms appear when commuting the operators). Furthermore, from the OPE relations immediately follows that the Galilean $W_3$ is not an extension of $W_3$. Still, there is a "commutative double" which makes handling this algebra somewhat easier.

The aim of this paper is to give a mathematically rigorous def{}inition of Galilean $W_3$ algebra and to initiate the study of highest weight representations. We give this def{}inition by using Kac - De Sole language of non-linear conformal Lie algebras (NLCA) in Section \ref{defs}. The universal enveloping vertex algebra of presented NLCA is precisely the Galilean $W_3$ VOA from \cite{AB} (up to normalisation). By choosing a suitably ordered basis of the Verma module we utilise the commutative part and reduce the problem of f{}inding zeroes of determinant formula to a very simple matrix (\ref{kriterijirr}) of rank 2. This enables classif{}ication of irreducible Verma modules. As in the case of GCA, reducibility depends only on highest weights and central charge corresponding to the action of commuting generators (cf.\ Theorem \ref{reduc}). By calculating singular and subsingular vectors of conformal weights 1 we give a basis of the vacuum module. Considering its determinant formula we also prove that (universal) Galilean $W_3$ is a simple algebra (Theorem \ref{prosta}). The method used in this section should be easily applied to other Galilean algebras, and we expect that analogous results hold in general. Much like in the case of Virasoro and GCA, the structure and representation theory of Galilean $W_3$ algebra seems to be rather dif{}ferent than that of $W_3$ (cf.\ \cite{BMP}, \cite{BW}, \cite{M}, \cite{W}). Notably, there are no minimal models, and the structure of the Verma modules seems to be uniform for all central charges.

It is well known that Quantum Drinfeld-Sokolov reduction of an (universal) af{}f{}ine VOA $V^k(\mathfrak{sl}_N)$ produces the $W$-algebra $W_N$. The resulting algebra can, in turn, be realised as a subalgebra of $M(1)_{N-1}$, the Heisenberg algebra of rank $N-1$, and $M(1)_{N-1}$ is constructed over a lattice $L_{N-1}$ with Gram matrix equal to the Cartan matrix of $\mathfrak{sl}_N$ (cf.\ \cite{RSW}). Since Galilean $W_N$ algebra is obtained from a tensor product of two copies of $W_N$, it is natural to consider its free f{}ield realisation within a product of two copies of Heisenberg algebras used in realisation of $W_N$. 
In Section \ref{ff} we start with a rank 4 lattice which is a product of two lattices whose Gram matrices are equal to Cartan matrices of $\mathfrak{sl}_3$. In the associated rank 4 Heisenberg algebra we detect a family of subalgebras isomorphic to Galilean $W_3$ algebras with arbitrary non-zero central charges. Furthermore by using the associated lattice VOA, we present a realisation of highest weight modules in Section \ref{ffr}. The highest weights are parametrised in such a way that reducibility of Verma modules corresponds to positive integral values of the f{}irst parameter (Proposition \ref{param}). This resembles the GCA case which is recalled in Subsection \ref{ff1}. We expect that the positive integral values of other parameters detect subsingular vectors in general. This is verif{}ied on (sub)singular vectors at level one (Example \ref{wt1}).

Throughout the paper we work with central charge $(c_L,c_M)\in\C^2$ such that $c_M=0$. In Appendix \ref{cm=0} we present the def{}inition of $\gw(c_L,0)$ which is an extension of Virasoro VOA by an ideal generated by the remaining three f{}ields.

\vspace{15pt}
The author is partially supported by the QuantiXLie Centre of Excellence, a project cof{}f{}inanced by the Croatian Government and European Union through the European Regional Development Fund - the Competitiveness and Cohesion Operational Programme (KK.01.\allowbreak1.\allowbreak1.\allowbreak01.\allowbreak0004). 

\vspace{10pt}
The author would like to thank Dra\v zen Adamovi\'c for useful comments and discussions and Simon Wood for bringing the OPE package for \textsc{Mathematica} to my attention.

\section{Def{}initions}\label{defs}
We start by recalling the notion of non-linear Lie conformal algebra introduced in \cite{DSK}.
\begin{definition}[\cite{DSK}]
A Lie conformal algebra is a $\C[D]$-module $R$ with a $\C$-linear map $[\ _\l\ ]:R\otimes R\to R[\l]$ satisfying the following axioms
\bea
&&[Da_\l b]=-\l[a_\l b]\qquad[a_\l Db]=(\l+D)[a_\l b],\label{sesq}\\
&&[a_\l b]=-[b_{-\l-D}a],\label{skew}\\
&&[a_\l[b_\mu c]]-[b_\mu[a_\l c]]-[[a_\l b]_{\l+\mu}c]=0.\label{jac}
\eea
\end{definition}
To any Lie conformal algebra $R$ one canonically associates $V(R)$, the universal enveloping vertex algebra of $R$ which is freely generated by $R$. For $a,b,c\in V(R)$ we have
\begin{eqnarray}
&&[a_\l b]=\operatorname{Res}_z e^{z\l} Y(a,z)b=\sum_{n\in\Zp}\frac{\l^n}{n!}a_{(n)}b\in\C[\l],\label{ope}\\
&&\normord{ab}-\normord{ba}=\int_{-D}^0[a_\l b]\operatorname{d\l},\\
&&\normord{(\normord{ab})c}-\normord{a(\normord{bc})}=\normord{a\int_0^D[b_\l c]\operatorname{d\l}}+\normord{b\int_0^D[a_\l c]\operatorname{d\l}}\label{qas}\\
&&[a_\l\normord{bc}]=\normord{[a_\l b]c}+\normord{b[a_\l c]}+\int_0^\l[[a_\l b]_\mu c]\operatorname{d\mu},\label{lwick}\\
&&[\normord{ab}_\l c]=\normord{(e^{D\partial_\l}a)[b_\l c]}+\normord{(e^{D\partial_\l}b)[a_\l c]}+\int_0^\l[b_\mu[a_{\l-\mu}c]]\operatorname{d\mu}.\label{rwick}
\end{eqnarray}
where $a_{(n)}b=\operatorname{Res}_z Y(a,z)b$ denotes the $n$-th product and $\normord{ab}=a_{(-1)}b$ is a normally ordered product of f{}ields $Y(a,z)$ and $Y(b,z)$. $\int$ is a formal def{}inite integral operator on $R[\l]$, i.e.\ $$\int_{a}^{b}\l^n\operatorname{d\l}=\frac{1}{n+1}(b^{n+1}-a^{n+1}),\qquad n\in\Zp.$$
Note that (\ref{ope}) encodes the commutator formula $[Y[a,z],Y(b,w)]$ also known as operator product expansion (OPE) $$a(z)b(w)\sim\sum_{n\in\N}\frac{(a_{(n)}b)(w)}{(z-w)^n}.$$
Inf{}inite-dimensional Lie algebras like Virasoro, Heisenberg and af{}f{}ine Kac-Moody algebras give rise to Lie conformal algebras whose universal enveloping algebras are precisely the universal vertex algebras associated to starting Lie algebras.
\begin{example}
If $R=\C[D]L\oplus\C$ (with $D1=0$) such that $[L_\l L]=(D+2\l)L+\frac{c}{12}\l^3$ then $V(R)=\operatorname{Vir}_c$.\\
Let $R_G=\C[D]L\oplus\C[D]M\oplus\C$ such that 
\bean
\ [L_\l L]&=&(D+2\l)L+\frac{c_L}{12}\l^3,\\
\ [L_\l M]&=&(D+2\l)M+\frac{c_M}{12}\l^3,\\
\ [M_\l M]&=&0.
\eean
Then $V(R_G)=L^{W(2,2)}(c_L,c_M)$ is GCA with central charge $(c_L,c_M)$.
\end{example}
However, many freely generated vertex algebras are not universal envelopes of f{}inite Lie conformal algebras because the $n$-th products of some of their generators are nonlinear, i.e.\ they contain normally ordered products. For this reason one needs to extend the $\l$-bracket to $\mathcal{T}(R)$, the tensor algebra of $R$.

For $a\in R$ and $B\in\mathcal{T}(R)$ def{}ine $\normord{aB}=a\otimes B$ so that $D(1)=0$, $D(\normord{AB})=\normord{(DA)B}+\normord{AD(B)}$ for $A,B\in\mathcal{T}(R)$ and then extend the $\l$-bracket to $[\ _\l\ ]:R\otimes R\to\C[\l]\otimes\mathcal{T}(R)$ using (\ref{qas}-\ref{rwick}). In order to deal with the Jacobi identity (\ref{jac}) we assume that $R$ is $\Z$-graded by conformal weights $R=\bigoplus_{\Delta\in\Zp}R[\Delta]$ such that
\bea
\Delta(Da)=\Delta(a)+1,\qquad\Delta(a_{(n)}b)=\Delta(a)+\Delta(b)-n-1\label{cw}
\eea
for $n\in\Zp$. Extending the grading to $\mathcal{T}(R)$ def{}ine the subspace $\mathcal{M}_\Delta(R)\subset \mathcal{T}(R)_{\leq\Delta}$ spanned by all elements
\bea
X\otimes(b\otimes c-c\otimes b)\otimes Y-X\otimes\left(\normord{\left(\int_{-D}^0[b_\l c]\operatorname{d\l}\right)Y}\right)
\eea
where $b,c\in R$, $X,Y\in\mathcal{T}(R)$ and $\Delta(X\otimes b\otimes c\otimes Y)\leq\Delta$.
\begin{definition}[\cite{DSK}]
A non-linear Lie conformal algebra (NLCA) is a $\Z$-graded $\C[D]$-module $R=\bigoplus_{\Delta\in\Zp}R[\Delta]$ with a $\C$-linear map $[\ _\l\ ]:R\otimes R\to\C[\l]\otimes\mathcal{T}(R)$ satisfying (\ref{sesq}-\ref{skew}), (\ref{cw}) and
\bea
&&\Delta([a_\l b])<\Delta(a)+\Delta(b),\\
&&[a_\l[b_\mu c]]-[b_\mu[a_\l c]]-[[a_\l b]_{\l+\mu}c]\in\C[\l,\mu]\otimes\mathcal{M}_{\Delta'}(R),
\eea
for all $a,b,c\in R$, where $\Delta'<\Delta(a)+\Delta(b)+\Delta(c)$.
\end{definition}
To each NLCA $R$ one associates a universal enveloping vertex algebra $V(R)=\mathcal T(R)/\mathcal M(R)$ which is freely generated by $R$. Conversely, if $V$ is a vertex algebra freely generated by a $\C[D]$-submodule $R\subset V$, then there is a NLCA structure on $R$ and $V\cong V(R)$. See \cite{DSK} for details.

\vspace{15pt}
The f{}irst example of NLCA comes from the well known Zamolodchikov $W_3$ algebra. VOA $W_3(c)$ is generated by a conformal f{}ield $\o(z)=\sum L(n)z^{-n-2}$ and a primary f{}ield $W(z)=\sum W(n)z^{-n-3}$ satisfying the following OPE:
\bea
\ \o(z)\o(w)&\sim&\frac{c/2}{(z-w)^4}+\frac{2\o(w)}{(z-w)^2}+\frac{\partial \o(w)}{z-w}\\
\ \o(z)W(w)&\sim&\frac{3W(w)}{(z-w)^2}+\frac{\partial W(w)}{z-w}\\
\ W(z)W(w)&\sim&\frac{c/3}{(z-w)^6}+\frac{2\o(w)}{(z-w)^4}+\frac{\partial\o(w)}{(z-w)^3}+\\
&+&\frac{1}{(z-w)^2}\left(\frac{3}{10}\partial^2\o(w)+2\b\L(w)\right)+\frac{1}{z-w}\left(\frac{1}{15}\partial^3\o(w)+\b\partial\L(w)\right)
\eea
where $\L(z)=\normord{\o(z)^2}-\frac{3}{10}\partial^2\o(z)$ and $\b=\frac{16}{22+5c}$.

Let $R=\C[D]L\oplus\C[D]W\oplus\C$ be a NLCA with the following $\l$-brackets:
\bea
\left[L_\l L\right]&=&(D+2\l)L+\frac{c}{12}\l^3,\\
\left[L_\l W\right]&=&(D+3\l)W,\\
\left[W_\l W\right]&=&\frac{c}{360}\l^5+\left(\frac{\l^3}{3}+\frac{\l^2}{2}D+\frac{3\l}{10}D^2+\frac{1}{15}D^3\right)L+\\
&&+\frac{16}{5c+22}(D+2\l)\left(L^2-\frac{3}{10}D^2L\right).\nonumber
\eea
Then $V(R)$ is precisely $W_3(c)$ (cf.\ \cite{DSK}).\\

Now we def{}ine the Galilean $W_3$ algebra.
\begin{definition}\label{gw3}
Let $c_L,c_M\in\C$, $c_M\ne0$. The Galilean $W_3$ NLCA is def{}ined as $$\gw(c_L,c_M)=\C[D]L\oplus\C[D]W\oplus\C[D]M\oplus\C[D]V\oplus\C,$$
where $\Delta(L)=\Delta(M)=2$, $\Delta(W)=\Delta(V)=3$ and with the following non-trivial $\l$-brackets
\begin{eqnarray}\label{lambda-b}
\left[L_\l L\right]&=&(D+2\l)L+\frac{c_L}{12}\l^3,\\
\left[L_\l M\right]&=&(D+2\l)M+\frac{c_M}{12}\l^3,\\
\left[L_\l W\right]&=&(D+3\l)W,\\
\left[L_\l V\right]&=&(D+3\l)V,\\
\left[M_\l W\right]&=&(D+3\l)V,\\
\left[W_\l W\right]&=&\frac{c_L}{360}\l^5+\left(\frac{\l^3}{3}+\frac{\l^2}{2}D+\frac{3\l}{10}D^2+\frac{1}{15}D^3\right)L+\\
&+&\frac{32}{5c_M}(D+2\l)\left(LM-\frac{3}{10}D^2M\right) - \frac{16}{5c_M^2}\left(c_L+\frac{44}{5}\right)(D+2\l)M^2,\nonumber\\
\left[W_\l V\right]&=&\frac{c_M}{360}\l^5+\left(\frac{\l^3}{3}+\frac{\l^2}{2}D+\frac{3\l}{10}D^2+\frac{1}{15}D^3\right)M+\frac{16}{5c_M}(D+2\l)M^2.
\end{eqnarray}
\end{definition}
Proving that the axioms of NLCA (in particular the Jacobi identity) hold is a straightforward, but rather tedious task (see Appendix \ref{app}). Another way of showing that this def{}inition is consistent is by obtaining a free f{}ield realisation. This is presented in Section \ref{ff}.

For simplicity, we shall use the same notation $\gw(c_L,c_M)$ for the associated universal enveloping vertex algebra which is generated by f{}ields
\bea
\o(z) &=& \sum_{n\in\Z} L(n)z^{-n-2},\\
W(z) &=& \sum_{n\in\Z} W(n)z^{-n-3},\\
M(z) &=& \sum_{n\in\Z} M(n)z^{-n-2},\\
V(z) &=& \sum_{n\in\Z} V(n)z^{-n-2}.
\eea
Also def{}ine the following f{}ields
\bea
\L(z)&=&\normord{L(z)M(z)}-\tfrac{3}{10}\partial^2 M(z)=\sum_{n\in\Z} \L(n)z^{-n-4},\\
\Theta(z)&=&\normord{M(z)M(z)}=\sum_{n\in\Z} \Theta(n)z^{-n-4},
\eea
so
\bea
\L(k)&=&\sum_{i\in\Z}\normord{L(-i)M(k+i)}-\frac{3}{10}(k+2)(k+3)M(k),\\
\Theta(k)&=&\sum_{i\in\Z}\normord{M(i)M(k-i)}.
\eea
Then the components of these f{}ields satisfy the following non-trivial commutation relations:
\begin{eqnarray}
\ [L(n),L(m)]&=&(n-m)L(n+m)+\d_{n+m,0}\frac{n(n^2-1)}{12}c_L\label{b1}\\
\ [L(n),W(m)]&=&(2n-m)W(n+m)\\
\ [L(n),M(m)]&=&(n-m)M(n+m)+\d_{n+m,0}\frac{n(n^2-1)}{12}c_M\label{b3}\\
\ [L(n),V(m)]&=&(2n-m)V(n+m)\label{b4}\\
\ [M(n),W(m)]&=&(2n-m)V(n+m)\label{b5}\\
\ [W(n),W(m)]&=&\frac{n-m}{30}\left((2n^2+2m^2-nm-8)L(n+m)+\frac{192}{c_M}\L(n+m)+\right.\label{ww}\\
	&&\left. -\frac{96}{c_M^2}(c_L+\frac{44}{5})\Theta(n+m)\right)+\d_{n+m,0}\frac{n(n^2-1)(n^2-4)}{360}c_L\nonumber\\
\ [W(n),V(m)]&=&\frac{n-m}{30}\left((2n^2+2m^2-nm-8)M(n+m)+\frac{96}{c_M}\Theta(n+m)\right)+\label{b2}\label{b7}\\
	&&+\d_{n+m,0}\frac{n(n^2-1)(n^2-4)}{360}c_M.\nonumber
\end{eqnarray}
This agrees with algebra introduced in \cite{AB} up to a normalisation factor $1/30$ (as in \cite{BJMN}).
Notice that $\o(z)$ and $M(z)$ generate a subalgebra isomorphic to the Galilean conformal algebra with central charge $(c_L,c_M)$. However, $\o(z)$ and $W(z)$ do not generate a copy of $W_3$ due to (\ref{ww}). A natural question arises: can we def{}ine a Galilean algebra in such a way that both Virasoro, and $W_3$ are its subalgebras acting on a commutative part? It turns out that this is not possible. Due to non-linearity, the Jacobi identity for such $\l$-brackets would not hold (see Appendix \ref{app2}).

\begin{corollary}
We have
\bea
\ch\gw(c_L,c_M)=(1-q^2)^{2}\prod_{n\geq3}(1-q^n)^{-4}.
\eea
\end{corollary}
\begin{proof}
We f{}ix an ordering $V>M>W>L$, and obtain a PBW basis of the universal enveloping vertex algebra $\gw(c_L,c_M)$ (cf.\ \cite{DSK}) which consists of monomials
\bea
V(-p_v)\cdots V(-p_1)M(-r_m)\cdots M(-r_1)W(-s_m)\cdots W(-s_1)L(-t_l)\cdots L(-t_1)\mathbf{1}\label{PBW-VOA}
\eea
such that $p_{k+1}\geq p_k\geq 3$, $r_{k+1}\geq r_k\geq 2$, $s_{k+1}\geq s_k\geq 3$, $t_{k+1}\geq t_k\geq 2$. Then the $q$-character formula is
\bea
\ch L(\c,0)=(1-q^2)^2(1-q)^4\prod_{n\geq1}\frac{1}{(1-q^n)^4}
\eea
which proves the assertion.
\end{proof}

\section{Highest weight modules}\label{mods}
Let $M$ be an ordinary module over the VOA $V$. In particular $M=\bigoplus_{h\in\C}M_h$, where $M_h=\lbrace m\in M:L(0)v=hv\rbrace$ is a subspace of conformal weight $h$, $\dim M_h<\infty$ and for each $u\in V$ and $v\in M$ we have $u(n)v=0$ for $n\gg0$.\\
A homogeneous vector $v\in M_h$ is called
\begin{itemize}
\item \textbf{singular} if for each $u\in V$ and $n\in\Zp$ we have $u(n)v=\delta_{n,0}h_u v$ for $h_u\in\C$;
\item \textbf{pseudo-singular} if for each $u\in V$ and $n\in\N$ we have $u(n)v=0$;
\item \textbf{subsingular} if there exists a submodule $N\subset M$ such that $v+N$ is singular in a quotient module $M/N$.
\end{itemize}
If $M$ is generated by a singular vector $v$ we say that $M$ is a \textbf{highest weight module}, and call $v$ a highest weight vector. Assume that $\gw(c_L,c_M)$-module is generated by a pseudo-singular vector, so $M=\bigoplus_{n\in\Zp}M_{h+n}$ for some $h\in\C$. Since $L(0)$, $M(0)$, $W(0)$ and $V(0)$ are mutually commuting operators acting on (a f{}inite-dimensional complex space) $M_h$, there exists a common eigenvector i.e.\ a highest weight vector. We restrict our study to highest weight modules, i.e.\ the case when $M_h$ is one-dimensional.

Let $\h:=(h_L,h_W,h_M,h_V)\in\C^4$ be arbitrary scalars and $\c:=(c_L,c_M)$. Verma module denoted by $V(\c,\h)$ is the universal highest weight module $\gw(c_L,c_M).v_\h$ of highest weight $\h$. The action of $\gw(c_L,c_M)$ on $v_\h$ is determined by relations
\bea
&&L(n)v_\h=\delta_{n,0}h_Lv_\h,\quad W(n)v_\h=\delta_{n,0}h_Wv_\h,\\ &&M(n)v_\h=\delta_{n,0}h_Mv_\h,\quad V(n)v_\h=\delta_{n,0}h_Vv_\h,\quad n\in\Zp,
\eea
and (\ref{b1})-(\ref{b7}).
\begin{remark}
Verma modules in classical case are def{}ined either as a quotient of universal enveloping algebra of a given Lie algebra $\mathfrak g$, or equivalently, as a module induced from Borel subalgebra, i.e.\ using a triangular decomposition $\mathfrak g=\mathfrak g_-\oplus\mathfrak g_0\oplus\mathfrak g_+$. Since $\gw(c_L,c_M)$ is not a (linear) Lie algebra and does not have a natural triangular decomposition (to subalgebras) we do not have these tools available. However, highest weight theory still works for general VOA. One way of proving existence of universal highest weight modules is by applying Zhu's theory. We sketch the idea without going into details.

It is well known in vertex algebra theory that for each VOA $V$ there exists an associative algebra $A(V)$ called Zhu's algebra of $V$ which controls the representation theory of $V$ in the following sense. For a $V$-module $M=\bigoplus_{n\in\Zp}M_{h+n}$, $M_h$ is an $A(V)$-module. Conversely, every $A(V)$-module is a top level of some $V$-module. Obviously, one-dimensional $A(V)$-modules correspond to highest weight $V$-modules.

It is not dif{}f{}icult to show that $A(\gw(c_L,c_M))$ is a commutative algebra with 4 generators. We will show in Section \ref{ff} that $\gw(c_L,c_M)$ is a subalgebra of a rank 4 Heisenberg algebra $M(1)$. Highest weight $M(1)$-modules then provide a realisation of highest weight $\gw(c_L,c_M)$-modules. The top levels of these modules are precisely the one-dimensional $A(\gw(c_L,c_M))$-modules and their existence then yields the existence of the Verma modules over $\gw(c_L,c_M)$ as universal highest weight modules.
\end{remark}
Since $\gw(c_L,c_M)$ is a freely generated VOA with a natural PBW basis
$$\lbrace\normord{V(z)^{n_v}M(z)^{n_m}W(z)^{n_w}L(z)^{n_l}}\vert n_v,m_n,n_w,n_l\in\Zp\rbrace$$
(cf.\ \cite{DSK}) by universality of the Verma module we see that the set of monomials
\begin{eqnarray}\label{pbw1}
V(-i_v)\cdots V(-i_1)M(-j_m)\cdots M(-j_1)W(-k_w)\cdots W(-k_1)L(-n_l)\cdots L(-n_1)v_\h
\end{eqnarray}
such that $v,m,w,l\in\Zp$, $i_v\geq\cdots\geq i_1\geq1$, $j_m\geq\cdots\geq j_1\geq1$, $k_w\geq\cdots\geq k_1\geq1$ and $n_l\geq\cdots\geq n_1\geq1$ forms a basis of $V(\c,\h)$.
We have $$V(\c,\h)=\bigoplus_{n\in\Zp}V(\c,\h)_n,\qquad V(\c,\h)_n=\lbrace v\in V(\c,\h):L(0)v=(h_L+n)v\rbrace.$$
Let $P(n)$ denote the partition function on $\Zp$. Then 
\bea
\dim V(\c,\h)_n = \sum_{i,j,k\geq 0} P(i)P(j)P(k)P(n-i-j-k)
\eea
so
\bea
\ch V(\c,\h)=q^{h_L}\prod_{n\geq1}\frac{1}{(1-q^n)^4}.
\eea

Let $V$ be a VOA, $M=\oplus_{n\in\Zp}M_{h+n}$ an ordinary weight $V$-module, and $M^*=\oplus_{n\in\Zp}M_{h+n}^*$ its restricted dual. Let $\<\cdot,\cdot\>:M^*\times M\to\C$ denote the natural pairing and $Y_{M^*}:V\to\operatorname{End}M^*[[z,z^{-1}]]$ be a linear map such that
\[
\<Y_{M^*}(v,z)w',w\>=\<w',Y_M(e^{zL(1)}(-z^{-1})^{L(0)}v,z^{-1})w\>
\]
for $v\in V$, $w\in M$, $w'\in M^*$. Then $(M^*,Y_{M^*})$ is a $V$-module, called the contragredient of $M$ (cf.\ \cite{FHL}). In case of $V=\gw(c_L,c_M)$ we have
\bea
L(n)^*&=&L(-n),\\
W(n)^*&=&-W(-n),\\
M(n)^*&=&M(-n),\\
V(n)^*&=&-V(-n),\\
\Lambda(n)^*&=&\Lambda(-n),\\
\Theta(n)^*&=&\Theta(-n).
\eea
\begin{lemma}
Let $L(\c,\h)$ denote the irreducible quotient of $V(\c,\h)$. Then 
\[
L(\c,\h)^*=L(\c,\h^*)
\]
where $\h^*=(h_L,-h_W,h_M,-h_V)$.
\end{lemma}

Natural pairing with the contragredient module induces a symmetric non-degenerate invariant bilinear form on $V(\c,\h)$ such that
\[
\<v_\h\vert v_\h\>=1,\qquad\<x.v_\h\vert y.v_\h\>=\<v_\h\vert x^*y.v_\h\>.
\]
In order to classify irreducible Verma modules we need to consider the determinant formula associated to this form. Since $\< V(\c,\h)_n\vert V(\c,\h)_m\>=0$ for $n\ne m$ we focus on $\det\< V(\c,\h)_n\vert V(\c,\h)_n\>$. We are only interested in its zeros so we will not calculate exponents of all the dif{}ferent factors in this determinant. Instead we introduce an ordering on the chosen basis of $V(\c,\h)_n$ and decompose the matrix $\< V(\c,\h)_n\vert V(\c,\h)_n\>$ to a tensor product of block triangular matrices, thus reducing the problem to f{}inding determinant of much simpler matrices. In the following subsection we show that this problem ultimately reduces to calculation of determinant
\begin{eqnarray}\label{kriterijirr}
D_n:&=&\begin{vmatrix}
	\< L(-n) v_{\h}\vert M(-n)v_{\h} \> & \< L(-n)v_{\h} \vert V(-n)v_{\h} \>\\
	\< W(-n) v_{\h}\vert M(-n)v_{\h} \> & \< W(-n)v_{\h} \vert V(-n)v_{\h} \>
\end{vmatrix}\\
	&=&\left(\frac{64}{5c_M}\left(h_M+\frac{n^2-1}{24}c_M\right)^2\left(h_M+\frac{n^2-4}{96}c_M\right)-9h_V^2\right)n^2. 
\end{eqnarray}
Furthermore, we use the same method to describe the module $L(\c,0)$ and prove simplicity of $\gw(c_L,c_M)$ in Subsection \ref{simple}.

\subsection{Determinant formula and classif{}ication of irreducible Verma modules}\label{det-form}

Instead of standard PBW basis (\ref{pbw1}) we will work with the following basis of $V(\c,\h)_n$:
\begin{eqnarray}\label{pbw}
B_n&=&\left\lbrace V(-n)^{v_n} M(-n)^{m_n}\cdots V(-1)^{v_1} M(-1)^{m_1} W(-n)^{w_n} L(-n)^{l_n}\cdots\right.\\
&&\left.\ \cdots W(-1)^{w_1}L(-1)^{l_1} v_{\h} : \sum_{i=1}^{n}i(v_i+m_i+w_i+l_i) = n\right\rbrace\nonumber
\end{eqnarray}

Now we introduce the ``commutative degree'' of a basis monomial: for $x\in B_n$ let
\bea
\deg_c x&=&\sum_{i=1}^{n}i(v_i+m_i),\\
V(\c,\h)_n^k&=&\mbox{span}_\C\lbrace x\in B_n:\deg_c x=k\rbrace,\\
B_n^k&=&B_n\cap V(\c,\h)_n^k.
\eea
Then
\bea
&&V(\c,\h)_n=\bigoplus_{k=0}^{n}V(\c,\h)_n^k\\
&&\dim V(\c,\h)_n^k = P_2(k)P_2(n-k) = \dim V(\c,\h)_n^{n-k},
\eea
where $P_2(m)=\sum_{i=0}^{m}P(i)P(m-i)$.\\
For $x\in B_n$ we write $x=x^c x^{nc} v_\h$, where $\deg_c x^c=\deg_c x$, and $\deg_c x^{nc}=0$. In other words, $x^c$ is a product of factors $M(-i)$ and $V(-i)$, while $x^{nc}$ is a product of factors $L(-i)$ and $W(-i)$. Then
\begin{eqnarray}
	\<x\vert y\> &=& \<(y^c)^*x^{nc}v_\h\vert(x^c)^*y^{nc}v_\h\>\label{sp1}
\end{eqnarray}
so $\<V(\c,\h)_n^k\vert V(\c,\h)_n^l\>=0$ if $k+l>n$. We may order the elements of $B_n$ so that for $x,y\in B_n$ $x\prec y$ if $\deg_cx<\deg_cy$. Then the Gram matrix of $\< V(\c,\h)_n\vert V(\c,\h)_n\>$ is block triangular with (nontrivial) diagonal blocks $\< V(\c,\h)_n^k\vert V(\c,\h)_n^{n-k}\>$, $k=0,\ldots,n$.\\
Let $x\in V(\c,\h)_n^k$ and $y\in V(\c,\h)_n^{n-k}$. Then
\begin{eqnarray}
\<x\vert y\> &=& \<(y^c)^*x^{nc}v_\h\vert v_\h\>\cdot\< v_\h\vert(x^c)^*y^{nc}v_\h\>\nonumber\\
	&=&\<x^{nc}v_\h\vert y^c v_\h\>\cdot\<x^cv_\h\vert y^{nc}v_\h\>.\label{sp2}
\end{eqnarray}
This shows that the Gram matrix of $\< V(\c,\h)_n^k\vert V(\c,\h)_n^{n-k}\>$ is a tensor product of matrices of the type $\< V(\c,\h)_{n-k}^0\vert V(\c,\h)_{n-k}^{n-k}\>$ and $\< V(\c,\h)_k^k\vert V(\c,\h)_k^0\>$. Therefore the problem of calculating (the zeroes of) $\det\< V(\c,\h)_n\vert V(\c,\h)_n\>$ reduces to f{}inding determinant of matrix $A_0=\< V(\c,\h)_n^0\vert V(\c,\h)_n^n\>$ which represents the action of monomials in $W(i)$ and $L(j)$ on monomials in $V(-k)$ and $M(-l)$.

Let us introduce a suitable ordering on $B_n^n$ and $B_n^0$ which makes $A_0$ block-triangular. We exploit the following fact: 
\begin{eqnarray}
&&k>i_v,j_m\Rightarrow X(k)V(-i_v)\cdots V(-i_1)M(-j_m)\cdots M(-j_1)v_\h=0,\qquad X\in\lbrace L,W\rbrace.\label{key}
\end{eqnarray}
Since we don't need to distinguish $L$ from $W$ and $M$ from $V$ we consider the type of monomial based only on partition. For $\bar k=(k_1,\ldots,k_n),\ \bar l=(l_1,\ldots,l_n)\in(\Zp)^n$ we def{}ine:
\bea
&&\bar k+\bar l = (k_1+l_1,\ldots,k_n+l_n)\\
&&\pp_n=\{(\bar k,\bar l)\in(\Zp)^n\times(\Zp)^n:\sum_{i=1}^{n}i(k_i+l_i)=n\}
\eea
and say that $(\bar k,\bar l)$ is of type $t(\bar k,\bar l)=\bar k+\bar l$.\\
There is a natural one to one correspondence between $\pp_n$ and $B_n^n$
\[
[VM](\bar v,\bar m) := V(-n)^{v_n} M(-n)^{m_n}\cdots V(-1)^{v_1} M(-1)^{m_1}v_{\h},
\]
i.e.\ between $\pp_n$ and $B_n^0$
\[
[WL](\bar w,\bar l) := W(-n)^{w_n} L(-n)^{l_n}\cdots W(-1)^{w_1}L(-1)^{l_1} v_{\h}.
\]
Def{}ine order on $(\Zp)^n$ by
\bea
\bar k\prec \bar l\quad\text{if}\quad k_{n-i}=l_{n-i}\text{ for }i=0,\ldots,j-1\text{ and }k_{n-j}>l_{n-j}
\eea
and use it to def{}ine order on $\pp_n$ by type: for $(\bar k,\bar l),(\bar k',\bar l')\in\pp_n$
\bea
(\bar k,\bar l) \prec (\bar k',\bar l')\quad&\text{if}&\quad t(\bar k,\bar l)\prec t(\bar k',\bar l')\\
	&\text{or}&\quad t(\bar k,\bar l)=t(\bar k',\bar l')\text{, and }\bar k\prec\bar k'.
\eea
This induces partial orders on $B_n^n$ and $B_n^0$. It is clear from (\ref{key}) that 
\bea
t(\bar k,\bar l) \prec t(\bar k',\bar l') \Rightarrow \<[WL](\bar k,\bar l)\vert[VM](\bar k',\bar l')\>=0
\eea
so $A_0$ is block-triangular with diagonal blocks of the type
\bea
\<V(\c,\h)_n^0(\bar k+\bar l)\vert V(\c,\h)_n^n(\bar k+\bar l)\>
\eea
where
\bea
V(\c,\h)_n^n (\bar k,\bar l)&=& \mbox{span}_\C\{[VM](\bar k',\bar l')\in B_n^n:t(\bar k',\bar l')=t(\bar k,\bar l)\},\\
V(\c,\h)_n^0 (\bar k,\bar l)&=& \mbox{span}_\C\{[WL](\bar k',\bar l')\in B_n^0:t(\bar k',\bar l')=t(\bar k,\bar l)\}
\eea
denote spans of basis elements of the same type. However, by construction we see that
\bea
\<V(\c,\h)_n^0(k_1,\ldots,k_n)\vert V(\c,\h)_n^n(k_1,\ldots,k_n)\>
\eea
is a tensor product of matrices of type
\bea
\<V(\c,\h)_n^0(0,\ldots,0,k_p,0,\ldots,0)\vert V(\c,\h)_n^n(0,\ldots,0,k_p,0,\ldots,0)\>,\quad p=1,\ldots,n
\eea
which correspond to the action of monomials $L(p)^rW(p)^{k_p-r}$, $r=0,\ldots,k_p$ on monomials \\$V(-p)^{k_p-r}M(-p)^r$, $r=0,\ldots,k_i$. Denote
\bea
&&W(p)V(-p)v_{\h}=a v_{\h},\label{abd}\\
&&W(p)M(-p)v_{\h}=L(p)V(-p)v_{\h}=b v_{\h},\label{abd1}\\
&&L(p)M(-p)v_{\h}=d v_{\h}.\label{abd2}
\eea

We show by induction on $n$ that for $i,j\in\lbrace 1,\ldots,n+1\rbrace$ the element at intersection of $i^{\text{th}}$ row and $j^{\text{th}}$ column of this matrix equals
\bea\label{relacija}
&&\a_{i,j}^{(n)}=L(p)^{i-1}W(p)^{n+1-i}V(-p)^{n+1-j}M(-p)^{j-1}v_\h=\nonumber\\
&&=(n+1-j)!(j-1)!a^{n-i-j+2}b^{i+j-2}\sum_{k=0}^{i-1}a^kd^k b^{-2k}\binom{n+1-i}{j-k-1}\binom{i-1}{k}v_\h.
\eea
The basis for $n=1$ gives (\ref{abd}-\ref{abd2}). Denote the righthand side of (\ref{relacija}) by $I_j$. Then
\bea
\a_{i,j}^{(n+1)}&=&(n+2-j)a I_j+(j-1)bI_{j-1}\\
&=&(n+2-j)!(j-1)!\times\\
&&\times\sum_{k=0}^{i-1}a^{n-i-j+k+3}b^{i+j-2k-2}d^k\binom{i-1}{k}\left(\binom{n+1-i}{j-k-1}+\binom{n+1-i}{j-k-2}\right)v_\h\nonumber
\eea
which proves the claim.\\
By direct calculation one can show that for each $j$ we have
\bea
&&\sum_{\ell=0}^i \binom{i-1}{\ell}\a_{i-\ell,j}(-b^2)^{\ell} = (n+1-j)!(j-1)!\binom{n+1-i}{j-i}a^{n-i-j+2}b^{j-i}(ad-b^2)^{i-1}
\eea
By elementary transformations we obtain a triangular matrix with determinant equal to
\bea
(ad-b^2)^{\frac{n(n+1)}{2}}\prod_{j=0}^{n}(n-j)!j!
\eea
The brackets (\ref{b3})-(\ref{b7}) yield
\bea
a &=& \frac{p}{15}\left((5p^2-8)h_M+\frac{96}{c_M}h_M^2\right)+\frac{p(p^2-1)(p^2-4)}{360}c_M,\\
b &=& 3p h_V,\\
d &=& 2p h_M+\frac{p(p^2-1)}{12}c_M.
\eea
From these considerations follows:
\begin{theorem}\label{reduc}
The Verma module $V(\c,\h)$ is reducible if and only if 
\begin{eqnarray}\label{krit-red}
h_V^2 = \frac{64\left(h_M+\frac{p^2-1}{24}c_M\right)^2\left(h_M+\frac{p^2-4}{96}c_M\right)}{45 c_M}
\end{eqnarray}
for some $p\in\N$. In that case, there is a singular vector in $V(\c,\h)^p_p$, where $p\in\N$ is the lowest such that (\ref{krit-red}) holds.
\end{theorem}
\begin{remark}
This method of f{}inding zeros of the determinant formula and thus classifying irreducible Verma modules relies on the fact that $\<M(z),V(z)\>$ is a commutative subalgebra of $\gw(c_L,c_M)$. Essentially the same method was used in classif{}ication of irreducibles over the GCA ($W(2,2)$) in \cite{DZ} and \cite{JZ} where it was shown that the Verma module $V^{W(2,2)}(h_L,h_M)$ is reducible if and only if 
\bea
\<L(-p)v\vert M(-p)v\>=p\left(2h_M+\frac{p^2-1}{12}c_M\right)=0
\eea
for some $p\in\N$. In that case there is a singular vector in $V^{W(2,2)}(h_L,h_M)_p^p$. Free f{}ield realisation of highest weight modules and formula for singular vectors was obtained in \cite{AR1}, \cite{AR2} and \cite{R}.\\
More importantly, this method can be applied to other Galilean W-algebras.
\end{remark}

\subsection{Simplicity of $\gw(c_L,c_M)$}\label{simple}
The submodule structure of reducible Verma module can be complicated. As an example we present formulas for (sub)singular vectors of weight $h_L+1$. In particular, we describe the vacuum module, and prove simplicity of $\gw(c_L,c_M)$.
\begin{example}\label{lvl1}
If $h_V^2=\frac{2h_M^2(32h_M-c_M)}{45c_M}$ (i.e.\ $p=1$) the singular vector from Theorem \ref{reduc} is given by
\bea
s.v_\h&=&
\left(V(-1)-\frac{3h_V}{2h_M}M(-1)\right)v_\h.
\eea
We have $W(0)s.v_\h=\left(h_W-\frac{3h_V}{h_M}\right)s.v_\h$, $M(0)s.v_\h=h_Ms.v_\h$, $V(0)s.v_\h=h_Vs.v_\h$.
Consider the determinant formula of the quotient module $V(\c,\h)/\<s.v_\h\>$. Factor corresponding to level one is
\bea
&\begin{vmatrix}
\<L(-1)v_\h\vert L(-1)v_\h\>&\<L(-1)v_\h\vert M(-1)v_\h\>&\<L(-1)v_\h\vert W(-1)v_\h\>\\
\<M(-1)v_\h\vert L(-1)v_\h\>&\<M(-1)v_\h\vert M(-1)v_\h\>&\<M(-1)v_\h\vert W(-1)v_\h\>\\
\<W(-1)v_\h\vert L(-1)v_\h\>&\<W(-1)v_\h\vert M(-1)v_\h\>&\<W(-1)v_\h\vert W(-1)v_\h\>
\end{vmatrix}=\\
&=\frac{8}{5}h_M^2\left(h_L-16\frac{h_M}{c_M}(3h_L+\frac{2}{5})+16\frac{h_M^2}{c_M^2}(c_L+\frac{44}{5})+3h_W\sqrt{\frac{5}{2c_M}(32h_M-c_M)}{}\right).
\eea
\begin{description}
\item[a] If $h_M=0$, then both $s_1^\pm.v_\h=\left(V(-1)\pm\frac{\iu}{\sqrt{10}}M(-1)\right)v_\h$ are singular vectors such that
\bean
(W(0)-h_W)s_1^\pm.v_\h = \pm\frac{2\iu}{\sqrt{10}}s_1^\pm.v_\h,
\eean
while $s_1.v_\h=M(-1)v_\h$ is a subsingular vector such that $s_1^\pm.v_\h\in\langle s_1.v_\h\rangle$.
\item[b] If
\begin{equation}\label{uvj}
h_L-16\frac{h_M}{c_M}(3h_L+\frac{2}{5})+16\frac{h_M^2}{c_M^2}(c_L+\frac{44}{5})+3h_W\sqrt{\frac{5}{2c_M}(32h_M-c_M)}=0
\end{equation}
then 
\bean
&&s_2.v_\h=\\
&&\left(W(-1)-\frac{3h_V}{2h_M}L(-1)-\frac{16}{5c_M}\left(\frac{h_M}{3h_Vc_M}(c_Mh_L-h_M(c_L-4))-\frac{3h_V}{h_M}\right)M(-1)\right)v_\h
\eean
is a subsingular vector in $V(\c,\h)$.
\item[ab] If $h_M=0$ and $h_L=3\iu\sqrt{5/2}h_W$ then
\bean
&&s_3.v_\h=\left(W(-1)+\frac{\iu}{\sqrt{10}}L(-1)\right)v_\h=\left(W(-1)-\frac{3 h_W}{2 h_L}L(-1)\right)v_\h
\eean
is a subsingular vector.
\item[abc] If $\h=0$, then $$s_4.v_0=L(-1)v_0$$
is a subsingular vector 
and $V(\c,0)_1\subset\<s_4.v_0\>$.
\end{description}
\end{example}
\begin{theorem}\label{prosta}
Vertex algebra $\gw(c_L,c_M)$ is simple for all $c_L,c_M\in\C$, $c_M\ne0$.
\end{theorem}
\begin{proof}
First we prove that $V(\c,0)/\< L(-1)v_0\>\cong \gw(c_L,c_M)$. 
From commutator relations follows that $V(-2)v,W(-2)v\in\< L(-1)v_0\>$ so the set $B'$ of monomials from $B_n$, $n\in\Zp$ such that $v_1=m_1=w_1=l_1=v_2=w_2=0$ spans $V(\c,0)/\< L(-1)v_0\>$. However this is precisely the PBW basis of universal enveloping VOA $\gw(c_L,c_M)$.

Now we follow the same procedure as in Subsection \ref{det-form}. One of the key facts in block diagonalisation of the Gram matrix of $V(\c,\h)$ is that $\dim V(\c,\h)_n^k = \dim V(\c,\h)_n^{n-k}$. The basis $B'$ retains this kind of symmetry so it is easy to see that determinant formula of the quotient module reduces to product of $\< L(-2)v_0\vert M(-2)v_0\>=\tfrac{c_M}{2}$ and determinants $D_n$, $n\in\Z_{>2}$ (\ref{kriterijirr}). However, for $\h=0$ we have $D_n=n(n^2-1)^2(n^2-4)c_M^2/4320$. This proves irreducibility, i.e.\ $\gw(c_L,c_M)\cong V(\c,0)/\< L(-1)v_0\>=L(\c,0)$.
\end{proof}

\begin{remark}
Proving simplicity of general VOAs is much more dif{}f{}icult. For example see \cite{AJM} for treatment of af{}f{}ine VOA, and the $W$-algebras obtained by the generalised quantised Drinfeld-Sokolov reduction (which includes $W_3$). Determinant formula for vacuum $W_3$-module is considered in \cite{CTW}.
\end{remark}

\section{Free f{}ield realisation}\label{ff}
We shall f{}irst recall the free f{}ield realisation of Galilean Virasoro algebra and its highest weight modules. This realisation was obtained using a rank 2 Heisenberg algebra and associated lattice VOA. Then we use the same idea to construct Galilean $W_3$ as a subalgebra of a rank 4 lattice VOA.
\subsection{Realisation of Galilean conformal algebra}\label{ff1}
Free f{}ield realisation of GCA or $W(2,2)$ was obtained in \cite{AR1,AR2,AR3} by means of embedding it in twisted Heisenberg-Virasoro algebra at level 0. Here we recall this construction (with slightly adjusted parametrisation).

Let $L=\Z c+\Z d$ be a rank 2 lattice such that 
\[
\< c\vert d\>=2,\quad\< c\vert c\>=\< d\vert d\>=0.
\]
Let $\mathfrak h=\C\otimes_\Z L$ and $\hat{\mathfrak h}=\mathfrak h\otimes\C [t,t^{-1}]\oplus\C K$ its af{}f{}inization. For any $h\in\mathfrak h$ we write $h(n)$ for $h\otimes t^n$ and we let $h(z)=\sum_{n\in\Z}h(n)z^{-n-1}$.

We denote by $M(1,h)$ the induced $\hat{\mathfrak h}$-module 
\[
U(\hat{\mathfrak h})\otimes_{U(\C[t]\otimes\mathfrak h\oplus\C K)}\C e^h
\]
such that $t\C[t]\otimes\mathfrak h$ acts trivially on $e^h$, $k(0)e^h=\< k\vert h\> e^h$ for $k\in\mathfrak h $ and $K e^h=e^h$. Then $M(1):=M(1,0)$ is a rank 2 Heisenberg vertex algebra generated by the f{}ields $h(z),\ h\in\mathfrak h$ and $M(1,h),\ h\in\mathfrak h$ are irreducible $M(1)$-modules. Furthermore, the f{}ields
\bean
\o(z)&=&\frac{1}{2}c(z)d(z)+\frac{c_L-2}{24}\partial c(z)-\frac{1}{2}\partial d(z),\\
M(z)&=&-\frac{c_M}{24}\left(c(z)^2-2\partial c(z)\right)
\eean
generate a vertex operator subalgebra of $M(1)$ which is isomorphic to GCA $L^{W(2,2)}(c_L,c_M)$. Def{}ine
\bean
\vpr&=&e^{-\frac{p+1}{2}d+\left((p+1)\frac{c_L-2}{24}-\frac{2p-r-1}{2}\right)c},\\
h_L[p,r]&=&(1-p^2)\frac{c_L-2}{24}+p\frac{2p-r-1}{2},\\
h_M[p]&=&\frac{1-p^2}{24}c_M.
\eean
Then we have
\bean
&&L(0)\vpr=h_L[p,r]\vpr,\qquad M(0)\vpr=h_M[p]\vpr,\\
&&h_L[-p,-r-2]=h_L[p,r],\qquad\quad h_M[-p]=h_M[p],\\
&&h_L[p,r]+p=h_L[p,r-2].
\eean
Denote by $\mathcal F_{p,r}=M(1).\vpr$. We f{}ix central charge $(c_L,c_M)$ and denote by $V[p,r]$ (resp.\ $L[p,r]$) the Verma (resp.\ irreducible) module of highest weight $(h_L[p,r],h_M[p])$. Then we have
\begin{theorem}[\cite{R},\cite{JZ},\cite{AR1},\cite{AR2}]\label{GCA}
Let $p>0$.
\begin{enumerate}
\item The Verma module $V[p,r]$ is reducible if and only if $p\in\N$. In that case, there is a singular vector $u'_p\in V[p,r]$ such that $\< u'_p\>\cong V[p,r-2]$.
\item $u'_p$ generates the maximal submodule in $V[p,r]$ if and only if $r\in\N$.
\item If $r\in\N$ then the maximal submodule in $V[p,r]$ is generated by a subsingular vector $u_{rp}$ of conformal weight $h_{p,r}+pr=h_{p,-r}$.
\item $\mathcal F_{p,r}\cong V[p,r]$ and $L[p,r]\cong U(W(2,2)).\vpr<\mathcal F_{-p,-r-2}\cong \mathcal F_{p,r}^\ast$.
\end{enumerate}
\end{theorem}
In the following we aim to obtain analogous results for $\gw$.

\begin{remark}
GCA is realised in \cite{AR1}, \cite{AR2}, \cite{AR3} as a subalgebra of Heisenberg-Virasoro VOA at level zero. One may also consider the $N=1$ super GCA. Realisation for central charge $c_L=11$ was presented in \cite{BJMN}. The $N=1$ super Heisenberg-Virasoro VOA was introduced in a recent paper \cite{AJR}, and the full treatment of level zero should appear soon as well. This will provide a natural framework for studying realisation of super GCA with arbitrary central charge.
\end{remark}

\subsection{Realisation of Galilean $W_3$ algebra}\label{realizacija_alg}
Let $L=\Z a+\Z b+\Z c+\Z d$ be a rank 4 lattice such that 
\[
\< a\vert b\>=\< c\vert d\>=-1,\quad\< a\vert c\>=\< a\vert d\>=\< b\vert c\>=\< b\vert d\>=0,\quad \< x\vert x\>=2,\quad x\in\lbrace a,b,c,d\rbrace.
\]
F{}ix $\l,\m\in\C$ such that $\l+i\m\neq0$ and let
\bean
\bar a&=& a+\iu c,\nonumber\\
\bar b&=& b+\iu d,\nonumber\\
\lm&=&\l+\iu\m.
\eean
Now we def{}ine the f{}ields in $M(1)$ which generate the Galilean $W_3$ algebra $\gw(c_L,c_M)$. Let
\bea
\o(z)&=&\frac{1}{3}\bigg(a(z)^2+a(z)b(z)+b(z)^2+c(z)^2+c(z)d(z)+d(z)^2\bigg)+\\
&&+\left.\l\partial a(z)+\l\partial b(z)+\m\partial c(z)+\m\partial d(z)\right.\nonumber\\
W(z)&=&\frac{\iu}{27\lm\sqrt{10}}\Bigg( 2((a(z)-b(z))(\bar a(z)+2\bar b(z))(2\bar a(z)+\bar b(z))+ \\
&&+(\bar a(z)-\bar b(z))(a(z)+2b(z))(2\bar a(z)+\bar b(z))+ \nonumber\\
&&+(\bar a(z)-\bar b(z))(\bar a(z)+2\bar b(z))(2a(z)+b(z)))+\nonumber\\
&&+9\l\bigg(\partial\bar a(z)(2\bar a(z)+\bar b(z))-\partial\bar b(z)(\bar a(z)+2\bar b(z))\bigg)+ \nonumber\\
&&+9\lm\bigg(\partial a(z)(2\bar a(z)+\bar b(z))-\partial b(z)(\bar a(z)+2\bar b(z))+\nonumber\\
&&+\partial\bar a(z)(2a(z)+b(z))-\partial\bar b(z)(a(z)+2b(z))\bigg)+ \nonumber\\
&&+18\l\lm\bigg(\partial^2\bar a(z)-\partial^2\bar b(z)\bigg)+9\lm^2\bigg(\partial^2 a(z)-\partial^2 b(z)\bigg)\Bigg)+\nonumber\\
&&\left(\frac{4}{15\lm^2}-\frac{\l}{\lm}-1\right)V(z)\nonumber\\
M(z)&=&\frac{1}{3}\bigg(\bar a(z)^2+\bar a(z)\bar b(z)+\bar b(z)^2\bigg)+\lm\bigg(\partial\bar a(z)+\partial\bar b(z)\bigg) \\
V(z)&=&\frac{\iu}{27\lm\sqrt{10}}\Bigg( 2(\bar a(z)-\bar b(z))(\bar a(z)+2\bar b(z))(2\bar a(z)+\bar b(z))+\\
&&+9\lm\bigg(\partial\bar a(z)(2\bar a(z)+\bar b(z))-\partial\bar b(z)(\bar a(z)+2\bar b(z))\bigg)+\nonumber\\
&&+9\lm^2\bigg(\partial^2\bar a(z)-\partial^2\bar b(z)\bigg)\Bigg).\nonumber
\eea
Direct calculation (with help of an OPE package for \textsc{Mathematica}) shows that f{}ields def{}ined above satisfy relations of Galilean $W_3$ algebra $\gw(c_L,c_M)$ with central charge
\bea
c_L &=& 4 - 24 (\l^2+\m^2),\\
c_M&=& - 24 \lm^2,
\eea
which gives us realisation for all $c_L,c_M\in\C$, $c_M\ne0$.

Notice that $L$ is a tensor product of two sublattices whose Gram matrices equal the Cartan matrix of $\mathfrak{sl}_3$. This kind of rank 2 lattice has been used in free f{}ield realisation of $W_3$ algebra (cf.\ \cite{RSW}).

We also remark that the f{}ields $M'(z)$ and $V'(z)$ which are obtained by substituting $\l$, $a(z)$ and $b(z)$ for $\bar\l$, $\bar a(z)$ and $\bar b(z)$ in $M(z)$ and $V(z)$ generate a copy of $W_3(-8\l-22/5)$ (cf.\ \cite{RSW}).

\section{Realisation of highest weight representations}\label{ffr}
Let $\C[L]$ be a group algebra of $L$ and $V_L=M(1)\otimes\C[L]$ associated VOA. We introduce a parametrisation of highest weight vectors in $V_L$. Let $e[p,q,r,s]$ denote a highest weight vector $e^k\in V_L$ where 
\bean
k&=&\left(\left(1+\frac{p+q}{2}\right)\l+\frac{2-r-s}{2\lm}\right)a+\left((1+q)\l-\frac{s-1}{\lm}\right)b+\nonumber\\
&+&\left(\left(1+\frac{p+q}{2}\right)\m+\iu\frac{2-r-s}{2\lm}\right)c+\left((1+q)\m-\iu\frac{s-1}{\lm}\right)d.
\eean
Weights $\h[p,q,r,s]$ of $e[p,q,r,s]$ are given by 
\begin{eqnarray}
h_L[p,q,r,s]&=&\frac{p(1-r)+3q(1-s)}{2}+\frac{c_L-4}{96}(4-p^2-3q^2),\label{paramhL}\\
h_W[p,q,r,s]&=&\frac{\iu}{2\sqrt{10}}\left(2p q (1-r)+(1-s)(p^2-3q^2)+q(p^2-q^2)\frac{52-5c_L}{120}\right),\label{paramhW}\\
h_M[p,q,r,s]&=&(4-p^2-3q^2)\frac{c_M}{96},\label{paramhM}\\
h_V[p,q,r,s]&=&\frac{\iu c_M}{48\sqrt{10}}q(q^2-p^2).\label{paramhV}
\end{eqnarray}
Direct calculation shows that 
\begin{eqnarray}\label{tezine}
\h[p,q,r,s]&=&\h[-p,q,-r+2,s]\\
	&=&\h\left[\frac{-p+3q}{2},-\frac{p+q}{2},\frac{-r+3s}{2},-\frac{r+s-4}{2}\right]\nonumber\\
	&=&\h\left[\frac{p+3q}{2},\frac{p-q}{2},\frac{r+3s-2}{2},\frac{r-s+2}{2}\right]\nonumber\\
	&=&\h\left[-\frac{p+3q}{2},\frac{p-q}{2},-\frac{r+3s-6}{2},\frac{r-s+2}{2}\right]\nonumber\\
	&=&\h\left[\frac{p-3q}{2},-\frac{p+q}{2},\frac{r-3s+4}{2},-\frac{r+s-4}{2}\right],\nonumber
\end{eqnarray}
i.e.\ parametrisation (\ref{paramhL}-\ref{paramhV}) is $S_3$-invariant under the action
\begin{eqnarray}
\sigma(p,q,r,s)&=&\left(\frac{-p+3q}{2},-\frac{p+q}{2},\frac{-r+3s}{2},-\frac{r+s-4}{2}\right)\label{act1}\\
\tau(p,q,r,s)&=&(-p,q,-r+2,s),\label{act2}
\end{eqnarray}
where $\sigma$ and $\tau$ are generators of $S_3$ of orders 3 and 2, respectively. 

\begin{proposition}\label{param}
Let $\mathcal{P=}\lbrace(p,q,r,s)\in\C^4:0<p<3q\rbrace$ where "$<$" denotes the lexicographical ordering on $\{(\Re(z),\Im(z)):z\in\C\}$. Let $\h\in\C^4$ such that $$h_V^2\neq\frac{64\left(h_M-\frac{c_M}{24}\right)^3}{45c_M}.$$
\item[i)]There exists a unique $(p,q,r,s)\in\mathcal P$ such that $\h=\h[p,q,r,s]$.
\item[ii)]For every $r\in\C$ we have
\begin{eqnarray}\label{krit1}
h_L[0,q,r,s]&=&\frac{3q}{2}(1-s)+\frac{c_L-4}{24}\left(1-3\left(\frac{q}{2}\right)^2\right),\\
h_W[0,q,r,s]&=&-\iu\sqrt{\frac{2}{5}}\left(3\left(\frac{q}{2}\right)^2(1-s)+\left(\frac{q}{2}\right)^3\frac{52-5c_L}{60}\right),\\
h_M[0,q,r,s]&=&\frac{c_M}{24}\left(1-3\left(\frac{q}{2}\right)^2\right),\\
h_V[0,q,r,s]&=&\iu\sqrt{\frac{2}{5}}\left(\frac{q}{2}\right)^3\frac{c_M}{12}.\label{krit4}
\end{eqnarray}
\item[iii)]We have $\h[p,q,r,s]^*=\h[p,-q,r,2-s]$, and $\sigma^2(p,-q,r,2-s)\in\mathcal P$.
\item[iv)]The Verma module $V[p,q,r,s]$ is reducible if and only if $p\in\N$. In that case there is a singular vector of conformal weight $h_L+p$ in $V[p,q,r,s]$.
\end{proposition}
\begin{proof}
i) From the Jacobian matrix of parametrisation (\ref{paramhL}-\ref{paramhV}) follows that $p=0$, and $p=\pm3q$ are critical values. Consider the action of $S_3=\<\sigma,\tau\>$ on $\C_4$ def{}ined by (\ref{act1}-\ref{act2}). Then $\mathcal P$ is a space of coinvariants $(\C^4)_{S_3}$ excluding critical values.\\
ii) and iii) direct calculation.\\
iv) is the reducibility condition (\ref{krit-red}) stated in terms of parametrisation.
\end{proof}

\begin{remark}
Note that the weights $\h$ not equal to (\ref{krit1}-\ref{krit4}) such that $h_V^2=\frac{64\left(h_M-\frac{c_M}{24}\right)^3}{45c_M}$ are not obtained by this parametrisation. This is analogous to realisation of GCA presented in Subsection \ref{ff1} where each $(0,r)$ produces weight $\left(\frac{c_L-2}{24},\frac{c_M}{24}\right)$. Highest weight modules of highest weights $\left(h_L,\frac{c_M}{24}\right)$ for $h_L\neq\frac{c_L-2}{24}$ were realised by means of deformed action on certain Whittaker modules (cf.\ \cite{AR3}) and by using the fact that these modules coincide with the highest weight modules over the Heisenberg-Virasoro algebra. We do not study realisation of remaining highest weights in this paper. However, it would be interesting to obtain these modules by some other means.
\end{remark}

\begin{example}\label{wt1}
Recall Example \ref{lvl1} of subsingular vectors at level 1. Then $s.e[1,q,r,s]$ is a singular vector in $\mathcal F_{1,q,r,s}$ while $s.e[-1,q,-r+2,s]=0$. Furthermore
\begin{description}
\item[a]$h_M[1,q,r,s]=0$ if and only if $q\in\lbrace\pm1\rbrace$.\\
$s_1.e[1,1,r,s]$ subsingular in $\mathcal F_{1,1,r,s}$, while $s_1.e[-1,-1,r,s]=0$.
\item[b](\ref{uvj}) holds if $r=1$.\\
$s_2.e[1,q,1,s]$ is subsingular in $\mathcal F_{1,q,1,s}$ and $s_2.e[-1,q,1,s]=0$.
\item[ab]$s_3.e[1,1,1,s]$ is subsingular in $\mathcal F_{1,1,1,s}$, and $s_3.e[-1,-1,1,s]=0$.
\item[abc]$s_4.e[1,1,1,1]$ is subsingular, and $s_4.e[-1,-1,1,1]=0$.
\end{description}
\end{example}
\begin{remark}
As we have seen (Proposition \ref{param} iv), integral values of $p$ detect positions of singular vectors in reducible Verma modules. Based on Example \ref{wt1} and on representation theory of GCA (Theorem \ref{GCA}) we expect that integral values of each of the remaining three parameters detect positions of subsingular vectors. Dif{}ferent sectors (of $S_3$ action on $\C^4$) should produce variant subquotients of $V[p,q,r,s]$, including the Verma module itself, and the irreducible quotient $L[p,q,r,s]$.
\end{remark}

\appendix
\section{$\l$-bracket calculation}
\subsection{Jacobi identity}\label{app}
Recall the Jacobi identity for $\l$-brackets (\ref{jac}). The most dif{}f{}icult calculation occurs in case $a=b=c=W$. We have (cf.\ \cite{DSK} Lemma 3.2)
\bean
\ [W_\l LM]&=&2 (DW)M+2L(DV)+3\l(WM+LV)+\left(4\l^2 D+\frac{5}{2}\l^3\right)V\\
\ [LM_\l W]&=&(D+3\l)(LV+WM)+2((DL)V+W(DM))+\frac{1}{2}\left(-3D^3-D^2\l+7D\l^2+5\l^3\right)V\\
\ [W_\l M^2]&=&4M(DV)+6\l MV\\
\ [M^2{}_\l W]&=&2(D+3\l)(MV)+4(DM)V
\eean
so $[W_\l [W_\mu W]]$ equals:
\bean
	&&\left(\frac{\mu^3}{3}+\frac{\mu^2}{2}(\l+D)+\frac{3\mu}{10}(\l+D)^2+\frac{1}{15}(\l+D)^3\right)(2D+3\l)W+\\
	&&+\frac{32}{5c_M}(2\mu+\l+D)\left((2D+3\l)(WM+LV)-2W(DM)-2(DL)V+\phantom{\frac{3}{10}}\right.\\
	&&\qquad+\left.(4\l^2D+\frac{5}{2}\l^3)V-\frac{3}{10}(\l+D)^2(2D+3\l)V\right)+\\
	&&-\frac{16}{5c_M^2}\left(c_L+\frac{44}{5}\right)(2\mu+\l+D)(4M(DV)+6\l MV);
\eean
and $[[W_\l W]_{\l+\mu}W]$ equals:
\bean
	&&\left(\frac{\l^3}{3}-\frac{\l^2}{2}(\l+\mu)+\frac{3\l}{10}(\l+\mu)^2-\frac{1}{15}(\l+\mu)^3\right)(D+3\l+3\mu)W+\\
	&&+\frac{32}{5c_M}(\l-\mu)\left((D+3\l+3\mu)(WM+LV)+2(DL)V+2W(DM)+\phantom{\frac{3}{10}}\right.\\
	&&+\left.\frac{1}{2}(3D^3-D^2(\l+\mu)7+D(\l+\mu)^2+5(\l+\mu)^3)V-\frac{3}{10}(\l+\mu)^2(D+3\l+3\mu)V\right)+\\
	&&-\frac{16}{5c_M^2}\left(c_L+\frac{44}{5}\right)(\l-\mu)(6(\l+\mu+D)MV-4M(DV)).
\eean
Comparing all the coef{}f{}icients one sees that (\ref{jac}) holds.
Similarly, in case $a=b=W$, $c=V$ we have
\bean
\ [LM_\l V]&=&3(D+\l)(MV)-2M(DV)
\eean
so
\bean
	\ [W_\l[W_\mu V]]&=&\left(\frac{\mu^3}{3}+\frac{\mu^2}{2}(\l+D)+\frac{3\mu}{10}(\l+D)^2+\frac{1}{15}(\l+D)^3\right)(2D+3\l)V+\\
	&+&\frac{16}{5c_M}(2\mu+\l+D)(6\l MV+4MDV)\\
	\ [[W_\l W]_{\l+\mu} V]&=&\left(\frac{\l^3}{3}-\frac{\l^2}{2}(\l+\mu)+\frac{3\l}{10}(\l+\mu)^2-\frac{1}{15}(\l+\mu)^3\right)(D+3\l+3\mu)V+\\
	&+&\frac{32}{5c_M}(\l-\mu)(3(D+\l+\mu)(MV)-2M(DV)).
\eean
Other cases are easier to check.

\subsection{Other def{}initions}\label{app2}
Suppose we want to construct a Galilean $W_3$ algebra in such a way that $W_3$ is its subalgebra, i.e.
\bean
\left[W_\l W\right]&=&\frac{c}{360}\l^5+\left(\frac{\l^3}{3}+\frac{\l^2}{2}D+\frac{3\l}{10}D^2+\frac{1}{15}D^3\right)L+\frac{16}{5c+22}(D+2\l)\left(L^2-\frac{3}{10}D^2L\right).
\eean
Let us check the Jacobi identity for a triple $W$, $W$, $M$. We see that $[W_\l[W_\mu M]]$ and $[W_\mu[W_\l M]]$ produce nonlinear terms in $M^2$, while $[[W_\l W]_{\l+\mu}M]$ produces $LM$, $(DL)M$ as well, so the identity (\ref{jac}) can not hold.

Therefore, either $[W_\l V]$ must contain a nonlinear term with factor $L$ (which means $W,L$ don't act on a commutative subalgebra generated by $M$ and $V$), or the nonlinear terms in $[W_\l W]$ must contain $M$ as a factor.

\section{The $c_M=0$ case}\label{cm=0}
Def{}inition \ref{gw3} of $\gw(c_L,c_M)$ does not allow for central charge $c_M=0$. If we rescale $W'(z)=c_M W(z)$ and then let $c_M=0$ we obtain the following non-trivial $\l$-brackets:
\begin{eqnarray}
\left[L_\l L\right]&=&(D+2\l)L+\frac{c_L}{12}\l^3,\\
\left[L_\l M\right]&=&(D+2\l)M,\\
\left[L_\l W'\right]&=&(D+3\l)W',\\
\left[L_\l V\right]&=&(D+3\l)V,\\
\left[W'_\l W'\right]&=&- \frac{16}{5}\left(c_L+\frac{44}{5}\right)(D+2\l)M^2,\\
\left[W'_\l V\right]&=&\frac{16}{5}(D+2\l)M^2.
\end{eqnarray}
The resulting VOA is an extension of Virasoro VOA by the ideal generated by primary f{}ields $M(z),W'(z),V(z)$.

We may follow the arguments of Subsection \ref{det-form} in obtaining the determinant formula. However in place of (\ref{abd1}) we have $W'(p)M(-p)v_\h=0$. Therefore 
\bea
\a_{i,j}^{(n)}=L(p)^{i-1}W'(p)^{n+1-i}V(-p)^{n+1-j}M(-p)^{j-1}v_\h=0
\eea
if $i<j$ which makes the matrix triangular with diagonal elements equal to
\bea
(n+1-i)!(i-1)!a^{n+1-i}d^{i-1}.
\eea
Since $a=\tfrac{32}{5}p h_M$ and $d=2p h_M$ we conclude that the Verma module $V(c_L,c_M=0,\h)$ is reducible if and only if $h_M=0$. In this case, $\langle M(-1)v_\h\rangle$ is a submodule, and reducibility of the associated quotient module corresponds to vanishing of $h_L$, $h_{W'}$ or $h_V$.


\end{document}